\documentclass[11pt]{paper}
\usepackage[latin1]{inputenc}
\usepackage[T1]{fontenc}
\usepackage{amsfonts,subcaption}
\usepackage{amsmath}
\usepackage{amssymb}
\usepackage{graphicx}
\usepackage{pstricks}
\usepackage{pst-grad}
\usepackage{multido}
\usepackage[vcentermath]{youngtab}
\usepackage{amsthm}
\usepackage{geometry}
\usepackage[absolute]{textpos}
\usepackage{url}

\definecolor{purple}{rgb}{0.8,0.12,0.8}
\definecolor{orange}{rgb}{1.0,0.7,0.0}
\definecolor{pink}{rgb}{1,0.5,0.8}
\definecolor{blackg}{rgb}{0.1,0.25,0.1}
\definecolor{ForestGreen}{cmyk}{0.91,0,0.88,0.42}
\definecolor{Turquoise}{cmyk}{0.85,0,0.20,0}

\newcommand{\cLR}{\mathcal{LR}}

\newcommand{\sg}{\mathcal{h}}
\newcommand{\sd}{\mathcal{i}}

\newcommand{\ba}{\mathbf{a}}

\newcommand{\bc}{\mathbf{c}}

\newcommand{\fB}{\mathfrak{B}}

\newcommand{\fJ}{\mathfrak{J}}

\newcommand{\fO}{\mathfrak{O}}

\newcommand{\nZ}{\mathbb{Z}}

\newcommand{\nN}{\mathbb{N}}

\newcommand{\nC}{\mathbb{C}}

\newcommand{\ua}{{\boldsymbol{a}}}
\newcommand{\al}{\alpha}

\newcommand{\la}{\lambda}

\newcommand{\eps}{\varepsilon}

\newcommand{\ka}{\kappa}

\newcommand{\ov}{\overline}

\newcommand{\Ind}{\text{Ind}}
\newcommand{\Irr}{\text{Irr}}

{ \begin{list}%
	{$-$}%
	{\setlength{\labelwidth}{10pt}%
	 \setlength{\leftmargin}{1pt}%
	 \setlength{\itemsep}{.1cm}}}%
{ \end{list} }

\newcommand{\tle}{\unlhd}

\newcommand{\quand}{\quad\text{and}\quad}
\newcommand{\quor}{\quad\text{or}\quad}
\newcommand{\Up}{{\rm Up}}
\newcommand{\Down}{{\rm Down}}

\newcommand{\up}[1]{^{\text{#1}}} 
\newcommand{\ula}{{\boldsymbol{\la}}}
\newcommand{\ulambda}{{\boldsymbol{\la}}}
\newcommand{\uemptyset }{{\boldsymbol{\emptyset}}}
\newcommand{\umu}{{\boldsymbol{\mu}}}
\newcommand{\unu}{{\boldsymbol{\nu}}}

\newtheorem{Th}{Theorem}[section]

\newtheorem{Lem}[Th]{Lemma}

\newtheorem{Prop}[Th]{Proposition}
\newtheorem{Def-Prop}[Th]{Definition-Proposition}
\newtheorem{conj}[Th]{Conjecture}
\theoremstyle{definition}
\newtheorem{Def}[Th]{Definition}
\newtheorem{Exa}[Th]{Example}

\theoremstyle{remark}
\newtheorem{Rem}[Th]{Remark}

\def\SS{\scriptstyle}

\def\finl{~$\SS{\blacksquare}$}

\begin{document}

\title{Ordering Families using Lusztig's symbols in type $B$: the integer case}
\author{J\'er\'emie Guilhot and Nicolas Jacon\thanks{Both authors are partly supported by the ANR (Project 
No ANR-12-JS01-0003) and would like to thank the ICERM for their invitation to the program "Automorphic Forms, Combinatorial Representation Theory and Multiple Dirichlet Series" where part of this work was done.}}  
\date{}
\maketitle

\begin{abstract}
Let $\Irr(W)$ be the set of irreducible representations of a finite Weyl group $W$.  Following an idea from Spaltenstein, Geck has recently introduced a preorder $\leq_L$ on $\Irr(W)$  in connection with the notion of Lusztig families. In a later paper with Iancu, they have shown that in type $B$ (in  the asymptotic case and in the equal parameter case) this order coincides with the order on Lusztig symbols as defined by Geck and the second author in \cite{GJ}. In this paper, we show that this caracterisation extends to the so-called integer case, that is when the ratio of the parameters is an integer. 
\end{abstract}


\section{Introduction}

Let $W$ be a finite Weyl group together with a weight function $L$ and let $\Irr(W)$ be the set of irreducible representations of $W$ over $\nC$.  One the one hand, Lusztig has defined a function $\ba$ on the set $\Irr(W)$ (now known as the Lusztig $\ba$-function) which allows him to define a partition of $\Irr(W)$ into the so-called families. On the other hand, Kazhdan-Lusztig theory naturally yields a partial order $\leq_{\cLR}$ on $\Irr(W)$ which in turn gives rise to a partition of $\Irr(W)$. When the weight function  is equal to the length function, these two partitions turns out to be the same: the proof relies on some very deep geometric interpretation. It is conjectured that the two partitions should coincide in the general case of unequal paramaters. The notion of families plays a fundamental role in the work of Lusztig  on the characters of reductive groups over finite fields. It also naturally appears in the work of Geck  \cite{Gcell} on the cellular's structure of Iwahori-Hecke algebras. \\

Despite a fairly simple definition, the order $\leq_{\cLR}$ is particularly hard to handle. In order to understand it and to give a combinatorial description, Geck has introduced a new preorder $\preceq_L$  in \cite{Gorder} on $\Irr(W)$ which is defined by using only standard operations on the irreducible representation of $W$, such as truncated induction from parabolic subgroups or tensorisation with the sign representation. He then proved that, in the equal parameter case, this order coincides with the Kazhdan-Lusztig order. He then conjectured that this should also hold in the general case of unequal parameters. Later on, Geck and Iancu have studied the order $\preceq_L$ in type $B$ and they have given a complete combinatorial description of this order in the equal parameter case and in the asymptotic case, using the combinatorics of Lusztig symbols. Their result brings in a  generalization of the dominance order which, as far as we know, first appeared in \cite[\S 5.7.5]{GJ}. It should be noted that  such an order also naturally  appears in various contexts of the representation theory of Hecke algebras \cite{Ge} but also in the theory of canonical bases for Fock spaces \cite[Ch. 6]{GJ} or the representation theory of Cherednik algebras 
   \cite{CGG,L,L1}. \\
 
The main aim of this paper is to show that Geck and Iancu results remain valid in the integer case; that is when the parameters satisfy some integer condition. The proof relies on some combinatorial properties of Lusztig symbols which may be of independent interest. The paper is organised as follows. In section 2, we introduced the basic concepts such as the order $\preceq_L$, the notion of symbols and the dominance order on the set of symbols. In Section 3, we caracterise adjacency of two symbols for the dominance order. This will play a crucial role for the proof of the main result in Section 4.


\section{Ordering Lusztig families}
\subsection{Lusztig Families}\label{parl}
Let $W$ be a finite Coxeter group with generating set $S$ and let $L$ be a weight function on $W$, that is a function 
 $L:W\to \mathbb{N}$ such that $L(w_1 w_2)=L(w_1)+L(w_2)$ whenever $w_1,w_2\in W$ satisfy  $\ell(w_1 w_2)=\ell(w_1)+\ell(w_2)$. (Here $\ell$ denotes the usual length function.)
 Let $\Irr (W)$ be the set of  complex irreducible representations of $W$. Using the ``generic degrees'', 
 Lusztig \cite{Luequal} has defined a function 
 $$\begin{array}{rccl}
\ua:& \Irr(W)&\longrightarrow& \nZ\\
  &E&\longmapsto&\ua_E
\end{array}$$
which plays an important role in the representation theory of Weyl groups; see for example \cite[Ch. 1,2,3]{GJ}.
Using  this function, Lusztig \cite[\S 4.2]{Lu} has shown that the set $\Irr(W)$ can be 
naturally  partitioned into the so-called ``families''. In Example \ref{type-A} and Proposition \ref{a-fonc-typeB}, we give explicit formulae for the values of the  $\ba$-function on the irreducible representations of the Weyl groups of type $A$ and  $B$. \\

We now briefly recall the definition of "families". We use the following notation: if $J\subset S$, $M\in \Irr (W_J)$ and 
$E \in \Irr (W)$, we write $M \uparrow E$ if $E$ is a constituent of $\Ind_J^S (M)$ 
and we write $M\rightsquigarrow_L E$ if $M \uparrow E$ {\text and }$\ua_M=\ua_E$ (where $\ba_M$ is the value of the $\ba$-function within the group~$W_J$). 

 \begin{Def} When $W=\{1\}$ then there is only one family which consists in the unit representation. When $W\neq \{1\}$ then 
 $E \in \Irr (W)$ and $E' \in \Irr (W)$ are in the same family if  there exists a sequence 
 $$E=E_0,E_1,\ldots, E_m=E'$$
 in  $\Irr (W)$ such that for all $i=1,2,\ldots,m$, the following condition is satisfied: there exist a subset $I_i \subset S$ and two simple modules $M_i, M_i ' \in  \Irr (W_{I_i})$ which belong to the same family in $\Irr (W_{I_i})$ such that 
 \begin{itemize}
 \item either $M_i \rightsquigarrow_L  E_{i-1}$ and $M_i  ' \rightsquigarrow_L  E_{i}$,
 \item or $M_i \rightsquigarrow_L  E_{i-1}\otimes \varepsilon $ and $M_i  ' \rightsquigarrow_L  E_{i}\otimes \varepsilon$,
 \end{itemize}
 where $\varepsilon$ denotes the sign representation of $W$. 
 \end{Def}
 
\begin{Exa} 
\label{Dihedral-family}
Let $W=\sg s,t \sd$ be the Dihedral group of order 12, that is the element $st$ is of order 6. Let $\ell$ be the usual length function on $W$. We refer to \cite[\S 5.3.4,\S 6.3.5,\S 6.5.10]{GP} for details in the computations.
We have 
$\Irr(W)=\{1_W,\eps,\eps_1,\eps_2, \varphi_1,\varphi_2\}$ where 
\begin{itemize}
\item $1_W$ is the trivial representation,
\item $\eps$ the sign representation, i.e. $\eps(w)=(-1)^{\ell(w)}$,
\item $\eps_s,\eps_t$ are two linear representations defined by $\eps_x(w)=(-1)^{\ell_x(w)}$ where $\ell_x$ for $x=s,t$ is the number of $x$ in any reduced expression of $w$,
\item $\varphi_i$ for $i=1,2$ are defined by 
$$\varphi_i(st)=\begin{pmatrix} \xi^i&0\\0&\xi^i\end{pmatrix}\quand \varphi_i(s)=\begin{pmatrix} 0&1\\1&0\end{pmatrix}\quad \text{ where } \xi=e^{-\frac{i2\pi}{6}}.
$$
\end{itemize}
There are only 4 parabolic subgroups, namely $W_{\emptyset}=\{1\}$, $W_s=\sg s\sd$, $W_t=\sg t\sd$ and $W$ itself. One can check that we obtain the following relations, inducing the trivial representation of each of the parabolic subgroups to $W$:
$$\begin{array}{lllllc}
1_W^W&=1_W,\\
1_{W_s}^W&=1_W+\eps_1+\underset{i}{\sum}\varphi_i,\\
1_{W_t}^W&=1_W+\eps_2+\underset{i}{\sum}\varphi_i,\\
1^W_{W_{\emptyset}}&=1_W+\eps+\eps_1+\eps_2+2\underset{i}{\sum}\varphi_i.
\end{array}$$
Next we obtain the following values for Lusztig $\ba$-function 
$$\begin{array}{|c|c|c|c|c|c|c|}\hline
E\in\Irr(W)& 1_W&\eps_s&\eps_t&\varphi_1&\varphi_2&\eps\\\hline
\ba_E&0&1&1&1&1&6\\\hline
\end{array}\ .$$
From there, we see that the families are 
$$\{1_W\}, \{\eps_1,\eps_2,\varphi_1,\varphi_2\},\{\eps\}.\text{\finl}$$

\end{Exa}

Kazhdan-Lusztig theory allows one to define another preorder on irreducible representations that we will denote $\leq_{\cLR}$. More precisely, one can define a preorder on $\leq_\cLR$ on $W$ using the Hecke algebra associated to $W$ and this preorder yields a partition of $W$ into the so-called two-sided cells. Note that two-sided cells are naturally equipped with a partial order induced by $\leq_{\cLR}$. Next, each two-sided cell $\bc$ affords a representation of $W$, not necessarily irreducible, that we will denote $M_{\bc}$. It turns out that any irreducible representation appears in a unique cell representation and therefore, the preorder $\leq_{\cLR}$ on cells induces a preorder on $\Irr(W)$ (that we denote in the same way)  as follows: $E\leq_{\cLR} E'$ if and only if there exists two two-sided cells $\bc$ and $\bc'$ such that $\bc\leq_{\cLR}\bc'$, $E$ is a constituent of $M_{\bc}$ and $E'$ is a constituent of   $M_{\bc'}$. This preorder naturally induces a partition of $\Irr(W)$. \\

As shown in  \cite[Chapter 5]{Lu} these two partitions turn out to be the same in the equal parameter case. The proof relies on some deep geometric interpretation which does not exist in the general case of unequal parameters. 

\begin{conj}[Lusztig]
\label{LR-family}
Let $W$ be a finite Coxeter group together with a weight function $L$. The partition of $\Irr(W)$ into families agrees with the partition induced by the Kazhdan-Lusztig order $\leq_{\cLR}$.
\end{conj}
As far as unequal parameters are concerned, this conjecture has been verified by explicit computation in type $I_2(m)$ and $F_4$ by Geck. In type $B$, it holds in the asymptotic case; see next section.  

\subsection{The order $\preceq_L$}
In order to have a better understanding of the partial order $\leq_{\cLR}$ in Conjecture \ref{LR-family}, Geck has introduced \cite{Gorder} a preorder $\preceq_L$ (see the definition below) satisfying the following condition:
 $E,E'\in\Irr(W)$  lie in the same family if and only if $E\preceq_L E'$ and $E'\preceq_L E$. Then with Iancu, they studied this preorder in type $B$ and give a complete combinatorial description of it in the asymptotic case and in the equal parameter case. Finally, they deduced that conjecture \ref{LR-family} holds in the ``asymptotic case''. \\
 
 Let us recall the definition of $\preceq_L$.   
\begin{Def} 
When $W=\{1\}$, then $\Irr(W)$ contains only the unit representation which is related with itself. When $W\neq \{1\}$ then 
$E \in \Irr (W)$ and $E' \in \Irr (W)$ satisfy $E \preceq_L E'$ if and only if 
there exists a sequence 
 $$E=E_0,E_1,\ldots, E_m=E'$$
 in  $\Irr (W)$ such that for all $i=1,2,\ldots,m$, the following condition is satisfied: there exist a subset $I_i \subset S$ and two simple modules $M_i, M_i ' \in  \Irr (W_{I_i})$ satisfying $M_i \preceq_L M_i'$ in $\Irr (W_{I_i})$ 
  such that 
 \begin{itemize}
 \item either $E_{i-1}\uparrow \Ind_{I_i}^S (M_{i})$ and $M_i ' \rightsquigarrow_L  E_{i}$,
 \item or $E_i\otimes \eps \uparrow \Ind_{I_i}^S ( M_{i})$ and $M_i' \rightsquigarrow_L  E_{i-1}\otimes \varepsilon$,
 \end{itemize}
 where $\varepsilon$ denotes the sign representation.  We write  $E\simeq_L E'$ if we have $E\preceq_L E'$ and 
  $E'\preceq_L E$.
 \end{Def}
 
Directly from the definition of $\preceq_L$, one can easily check that 
\begin{enumerate}
\item  if $E$ and $E'$ belong to the same family, then $E\simeq_L E'$;
\item $E\preceq_L E' \Rightarrow \ua_E \geq \ua_{E'}$.
\end{enumerate}
The fact that this order gives rise to families (in other words that the converse of 1. holds) 
is not straightforward by any means: it is one of the main result of \cite{Gorder} and \cite{GL}. \\

\begin{Exa}
\label{type-A}
Let  us study the case where $W$ is a Weyl group of type $A_n$ with diagram as follows:
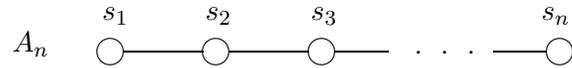
\begin{center}
\begin{picture}(250,30)
\put(  3, 08){$A_n$}
\put( 40, 08){\circle{10}}
\put( 45, 08){\line(1,0){30}}
\put( 80, 08){\circle{10}}
\put( 85, 08){\line(1,0){30}}
\put(120, 08){\circle{10}}
\put(125, 08){\line(1,0){20}}
\put(155, 05){$\cdot$}
\put(165, 05){$\cdot$}
\put(175, 05){$\cdot$}
\put(185, 08){\line(1,0){20}}
\put(210, 08){\circle{10}}
\put( 37, 20){$s_1$}
\put( 76, 20){$s_2$}
\put(116, 20){$s_3$}
\put(203, 20){$s_{n}$}
\end{picture}
\end{center}
The description of $\preceq_L$ is given in \cite[Ex. 2.11]{GL}. We explain it for the convenience of the reader as our proof in type $B_n$ will roughly 
 follows the same pattern. 
It is a well-known fact that the irreducible representations of $W$ are parametrised by partitions of $n$. We use the labelling as in \cite{GP} where, for instance, the unit representation is parametrised by the partition with one part equal to $n$ and the sign representation by the partition with $n$ parts equal to $1$. For a partition $\la$ of $n$, we denote by $E^{\la}$ the corresponding element of $\Irr(W)$. We will denote by $W_k$ the parabolic subgroup of $W$ generated by $\{s_1,\ldots,s_k\}$.\\

Let $\la=(\la_1,\ldots,\la_m)$ be a partition of $n$. Then we have:
$$\ba(E^{\la})=\sum_{i=1}^{m}(i-1)\la_i.$$
We claim that the order $\preceq_L$ is nothing else than the usual dominance order $\unlhd$ on partition. Recall that, for two partitions $\la$ and $\mu$ of $n$, we have $\la\unlhd \mu$ if and only if 
$$\sum_{i=1}^k \la_i\leq \sum_{i=1}^k \mu_i\text{ for all $k\geq 1$}.$$
We will need the following results and we refer to \cite{GP} for proofs and details.
\begin{enumerate}
\item[(1)] We have $E\preceq_L E'$ if and only if there exists a sequence $E=E_0,\ldots,E_m=E'$ such that for each $1\leq i\leq m$ there is a decomposition $n=k_i+\ell_i$ and two irreducible representations $M_i,M_i'\in\Irr(W_{k_i})$ such that $M_i\preceq_L M_i'$ and either
\begin{itemize}
\item $M_i\boxtimes\eps_{\ell}\uparrow E$ and $M'_i\boxtimes \eps_{\ell}\leadsto E'$
\item $M_i\boxtimes\eps_{\ell}\uparrow E'\otimes \eps$ and $M'_i\boxtimes \eps_{\ell}\leadsto E\otimes \eps$
\end{itemize}
\item[(2)] (Pieri's Rule) Let $n=k+\ell$ and let $\la$ be a partition of $k$ and $\mu$ a partition of $n$. Then 
$$E^\la\boxtimes \eps_\ell\uparrow E^\mu$$
if and only if $\mu$ can be obtained from $\la$ by adding one box to $\ell$ different parts of $\la$.
\item[(3)] (truncated Pieri's Rule) Let $n=k+\ell$ and let $\la$ be a partition of $k$ and $\mu$ a partition of $n$. Then 
$$E^\la\boxtimes \eps_\ell\leadsto E^\mu$$
if and only if $\mu$ can be obtained from $\la$ by adding one box to the $\ell$ greatest parts of $\la$.
\end{enumerate}
We are now ready to prove the claim. 
Let $\la,\mu$ be such that $E^{\la}\preceq_L E^{\mu}$. It is enough to consider an elementary step in the definition, so by (1), we may assume that there exists $k\leq n$ and  $E^{\la'},E^{\mu'}\in \Irr(W_k)$  such that the following is satisfied: $E^{\la'}\preceq_L E^{\mu'}$ within $W_k$ and we have either
 \begin{itemize}
\item $E^{\la'}\boxtimes\eps_{\ell}\uparrow E^\la$ and $E^{\mu'}\boxtimes \eps_{\ell}\leadsto E^\mu$;
\item $E^{\la'}\boxtimes\eps_{\ell}\uparrow E^\mu\otimes \eps$ and $E^{\mu'}\boxtimes \eps_{\ell}\leadsto E^\la\otimes \eps$.
\end{itemize} 
Arguying by induction we may assume that $\la'\unlhd \mu'$. Suppose that we are in the first case. Then, by (2), $\la$ can be obtained from $\la'$ by adding one box to $\ell$ different parts of $\la'$ while, by (3), $\mu$ can be obtain from $\mu'$ by adding one box to the $\ell$ greatest parts of $\mu$. It is clear that in this case we still have $\la\unlhd \mu$. The argument is similar in the second case, since $E^{\la}\otimes\eps=E^{\ov\la}$ where $\ov\la$ is the transposed partition of $\la$. \\
 
 Conversely, assume that $\la\unlhd \mu$. We may assume that $\la$ and $\mu$ are adjacent in the dominance order that is, if $\nu$ satisfies $\la\preceq\nu\preceq \mu$ then either $\nu=\la$ or $\nu=\mu$. By \cite[\S 1.16]{Mac-Do}, we know that $\mu$ differs from $\la$ by only one box. That is there exists $i_1<j_1\in\nN$ such that 
 $$\mu_{i_1}=\la_{i_1}+1,\quad \mu_{j_1}=\la_{j_1}-1\quand \mu_i=\la_i \text{ for all $i\neq i_1,j_1$}.$$
Let $\nu$ be the partition defined by 
$$\nu_k=\begin{cases}
\mu_k-1&\mbox{ if $k<j_1$}\\
\mu_k&\mbox{ if $k\geq j_1$}
\end{cases}$$
Then, one sees that $\mu$ can be obtained from $\nu$ by adding one box to the $j_1-1$ greatest parts of $\mu$ while $\la$ can be obtained by adding one box to $j_1-1$ different parts of $\nu$ (namely the part indexed by $\{1,\ldots,j_1\}-\{i_1\}$). Therefore we have $E^\nu\boxtimes \eps_{j_1-1}\uparrow E^\la$ and $E^\nu\boxtimes \eps_{j_1-1}\leadsto E^\mu$ and $\la\preceq_L\mu$ as required using (1). 
\finl

\end{Exa}

Finally, Geck conjectures the following relation.
\begin{conj}[Geck \cite{Gorder}]
 Let $E$ and $E'$ be in $\Irr (W)$ then we have $E\preceq_L E'$ 
 if and only if  $E\preceq_{\cal LR} E'$. 
\end{conj}
It has been shown in \cite[Prop. 3.4]{Gorder} that the implication $E\preceq_L E' \Rightarrow E \preceq_{\cal LR} E'$ always holds. 
The conjecture has been proved in the following cases
\begin{itemize}
\item In the equal parameter case by Geck \cite[Th. 4.11]{Gorder}. The proof relies on a very deep interpretation of Kazhdan-Lusztig theory in terms of geometry;
\item in type $F_4$ and $I_2 (m)$ for all choices of $L$ by Geck \cite[\S 3]{Gorder};
\item in type $B_n$, in the so called asymptotic case by Geck and Iancu \cite{GL}.
\end{itemize}


\subsection{Type $B_n$}
\label{type-B}
ILet $W$ be a the Weyl group of type $B_n$ with diagram as follows  
\begin{center}
\begin{picture}(250,30)
\put(  3, 08){$B_n$}
\put( 40, 08){\circle{10}}
\put( 44, 05){\line(1,0){33}}
\put( 44, 11){\line(1,0){33}}
\put( 81, 08){\circle{10}}
\put( 86, 08){\line(1,0){29}}
\put(120, 08){\circle{10}}
\put(125, 08){\line(1,0){20}}
\put(155, 05){$\cdot$}
\put(165, 05){$\cdot$}
\put(175, 05){$\cdot$}
\put(185, 08){\line(1,0){20}}
\put(210, 08){\circle{10}}
\put( 37, 20){$t$}
\put( 76, 20){$s_1$}
\put(116, 20){$s_2$}
\put(203, 20){$s_{n-1}$}
\end{picture}
\end{center}
together with a weight function $L:W\to \mathbb{Z}$. 
 \begin{Rem}\label{asy}
The case where $L(t)>(n-1)L(s_i)>0$ for all $i=1,\ldots,n-1$ 
  is called the ``asymptotic case''. 
 
 \end{Rem}

 In the rest of the paper, we will assume that we are in the so-called "integer case", that is,  $L(t)=r\in \mathbb{N}$ and $L(s_i)=1$ if $i=1,\ldots,n-1$; we will denote by $L_r$ such a weight function.
 We now describe how the irreducible representations, the $\ua$-function and the families can be described in this case, for all choices of weight function $L_r$. \\

 The set of irreducible representations of $W$ is naturally labelled by the set $\Pi_n^2$  of bipartitions of rank $n$. Recall that for a partition $\lambda=(\lambda_1,\ldots,\lambda_r)$ we have set $|\lambda|=\sum_{1\leq i\leq r }\lambda_i$.
A bipartition of $n$ is just a 
pair of partitions $\ulambda=(\lambda^{1},\lambda^{2})$ such that
$|\lambda^{1}|+|\lambda^{2}|=n$. We thus have
$$\Irr (W)=\{E^{\ulambda}\ |\ \ulambda\in \Pi_n^2\}.$$
We refer to \cite{GJ} for more details and for an explicit description of the irreducible representations $E^\ula$ for $\ula\in \Pi^2_n$. \\

To each bipartition, one can associate an important object: its symbol. The advantage of this new notion is that it takes into account the the weight function that we have choosen on $W$.  Let  $\ulambda=(\lambda^1,\lambda^2)$  be a bipartition of $n\in \mathbb{N}$ with $\la^i=(\la^i_1,\ldots,\la^i_{r_i})$ for $i=1,2$. 
  Let  $N\in \mathbb{N}$ be such that 
  $$N\geq \textrm{max} \{i\in \mathbb{N}\ |\ \textrm{max} \{\lambda^1_i,\lambda^2_i\}\neq 0\}.$$
  Such an integer will be call admissible for the bipartition $\ulambda$. Then, we associate to $\ulambda$ and $N$ the  
 $(b,N)${\it -symbol}
  $$
\mathfrak{B}_{(b,N)}  (\ulambda)=\left(
\begin{array}
[c]{lllll}
\mathfrak{B}^{2}_{N} &   \ldots  &    \mathfrak{B}^{2}_1 &  & \\ 
\mathfrak{B}^{1}_{N+b}  &   \ldots  &  \mathfrak{B}^{1}_{2}    & \ldots   &  \mathfrak{B}^{1}_1   \\  
\end{array}
\right)$$
where 
\renewcommand{\arraystretch}{1.5}
$$\begin{array}{rllcc}
\mathfrak{B}^1_j&:=\lambda^1_j-j+N+b\text{ for } j=1,\ldots,N+b, \\
\mathfrak{B}^2_j&:=\lambda^2_j-j+N\text{ for } j=1,\ldots,N.
\end{array}
$$ 
We will denote by 
 $\kappa_{(b,N)} (\ulambda)$ the sequence of $2N+b$ elements in $
\mathfrak{B}_{(b,N)}  (\ulambda)$ written in decreasing order. A straightforward computation shows that it is a partition of  the integer 
$$f(b,N,n)=n+(N(N-1))/2+(N+b)(N+b-1)/2.$$
We set $\kappa_{(b,N)} (\ulambda)=(\kappa_1,\ldots,\kappa_{2N+b})$ and 
 $$n_{(b,N)} (\ulambda)=\sum_{1\leq i\leq r} (i-1) \kappa_i.$$ 
\begin{Exa}
Let $b=2$, $\la^1=(5,1)$, $\la^2=(2,2,1)$ and $\ulambda= (\la^1,\la^2)$. Then, $N=3$ is admissible for $\ula$ and we have
\renewcommand{\arraystretch}{1}

$$
\mathfrak{B}_{(2,3)}  (\ulambda)=\left(
\begin{array}
[c]{llll l}
 1 & 3 & 4 &  & \\
  0 & 1 &2 & 4 &9
\end{array}
\right). $$
Then $\kappa_{(2,3)} (\ulambda)=(9,4,4,3,2,1,1,0)$ which is a partition of $f(2,3,11)=24$. \finl
\end{Exa}

\begin{center}
{\it From now on and until the end of this section, we fix a weight function $L_b$ on $W$. }
\end{center}

Using symbols, one can easily describe the $\ua$-function and the families in type $B_n$ (see  \cite[\S 22.14, \S 23.1]{Lu}).  
\begin{Prop}[Lusztig]
Let $\ulambda=(\lambda^1,\lambda^2),\umu=(\mu^1,\mu^2)\in \Pi^2_n$ and assume that  $N\in \mathbb{N}$ 
 is admissible for both $\ula$ and $\umu$. Then the irreducible representations $E^\ulambda$ and $E^\umu$ belong to the same family in $W$ with respect to the weight function $L_{b}$
  if and only if $\kappa_{(b,N)} (\ulambda)=\kappa_{(b,N)} (\umu)$.
  \end{Prop}
  \begin{Exa}
(1) As noted in \cite[Ex. 8.1]{GL}, if we have $b>n-1$ (the asymptotic case), then all the families are singletons. \\

(2) Consider the Weyl group of type $B_3$ together with the weight function $L_1$. There are $10$ bipartitions of $3$. In Figure 1, we list all the bipartitions of 3 together with their $(3,1)$-symbols and the partitions $\ka_{3,1}$.  
\renewcommand{\arraystretch}{1.7}
\begin{figure}[h!]
$$\begin{array}{|c|c|c|ccc}\hline
\text{Bipartition of 3}& \text{Symbol} & \ka_{1,3} \\\hline
\Yboxdim5pt
(\yng(1,1,1),\tiny{\emptyset})&\renewcommand{\arraystretch}{.75}\begin{pmatrix}0&1&2&\\0&2&3&4 \end{pmatrix}&(4,3,2,2,1,0,0) \\\hline\Yboxdim5pt
(\yng(2,1),\tiny{\emptyset})&\renewcommand{\arraystretch}{.75}\begin{pmatrix}0&1&2&\\0&1&3&5 \end{pmatrix}&(5,3,2,1,1,0,0) \\\hline\Yboxdim5pt
(\yng(3),\tiny{\emptyset})&\renewcommand{\arraystretch}{.75}\begin{pmatrix}0&1&2&\\0&1&2&6 \end{pmatrix}&(6,2,2,1,1,0,0) \\\hline\Yboxdim5pt
(\tiny{\emptyset},\yng(1,1,1))&\renewcommand{\arraystretch}{.75}\begin{pmatrix}1&2&3\\0&1&2&3 \end{pmatrix}&(3,3,2,2,1,1,0) \\\hline \Yboxdim5pt 
(\tiny{\emptyset},\yng(2,1))&\renewcommand{\arraystretch}{.75}\begin{pmatrix}0&2&4\\0&1&2&3 \end{pmatrix}&(4,3,2,2,1,0,0) \\\hline \Yboxdim5pt 
(\tiny{\emptyset},\yng(3))&\renewcommand{\arraystretch}{.75}\begin{pmatrix}0&1&5\\0&1&2&3 \end{pmatrix}&(5,3,2,1,1,0,0) \\\hline \Yboxdim5pt 
(\yng(1),\yng(1,1))&\renewcommand{\arraystretch}{.75}\begin{pmatrix}0&2&3&\\0&1&2&4 \end{pmatrix}&(4,3,2,2,1,0,0) \\\hline \Yboxdim5pt 
(\yng(1),\yng(2))&\renewcommand{\arraystretch}{.75}\begin{pmatrix}0&1&4&\\0&1&2&4  \end{pmatrix}&(4,4,2,1,1,0,0) \\\hline \Yboxdim5pt 
(\yng(1,1),\yng(1))&\renewcommand{\arraystretch}{.75}\begin{pmatrix}0&1&3&\\0&1&3&4 \end{pmatrix}&(4,3,3,1,1,0,0) \\\hline \Yboxdim5pt 
(\yng(2),\yng(1))&\renewcommand{\arraystretch}{.75}\begin{pmatrix}0&1&3&\\0&1&2&5  \end{pmatrix}&(5,3,2,1,1,0,0) \\\hline  
\end{array}$$
 \caption{Bipartitions and symbols for $(B_3,L_1)$ }
              \label{ka2}
\end{figure}

One can then easily check that the families are given by 
$$\left\{\Yboxdim5pt(\yng(1,1,1),\tiny{\emptyset}), (\tiny{\emptyset},\yng(2,1)),(\yng(1),\yng(1,1))\right\}, 
\left\{\Yboxdim5pt (\yng(2,1),\tiny{\emptyset}), (\tiny{\emptyset},\yng(3)),(\yng(2),\yng(1))\right\}, 
\left\{\Yboxdim5pt (\yng(3),\tiny{\emptyset})\right\}, 
\left\{\Yboxdim5pt  (\tiny{\emptyset},\yng(1,1,1))\right\}, 
\left\{\Yboxdim5pt (\yng(1),\yng(2))\right\},
\left\{\Yboxdim5pt (\yng(1,1),\yng(1))\right\}.
\text{\finl}$$

\end{Exa}

\begin{Prop}[Lusztig]
\label{a-fonc-typeB}
Let $\ulambda=(\lambda^1,\lambda^2)\in \Pi^2_n$ and assume that  $N\in \mathbb{N}$ 
 is admissible for~$\ula$.  Then  $$\ua(E^{\ulambda})=n_{(b,N)} (\ula)-n_{(b,N)} (\uemptyset)$$
  \end{Prop}
Note that there is a similar description of the $\ba$-function in the case the more general case where $W$ is the  complex reflection group 
 $ (\nZ/\ell\nZ)^n\rtimes \mathfrak{S}_n $ (see \cite[Prop 5.5.11]{GJ}).\\
 
Recall that our main problem is an explicit description of the order $\preceq_L$. We already know 
  that this preorder is a refinement of the preorder induced by the $\ua$-function.
Actually, 
   looking at the above formula for the $\ua$-function, we see that  the order induced by this function admits another natural refinement 
    which has first been introduced in \cite{GJ}, again in the wider context of complex reflection groups of types 
 $ (\nZ/l\nZ)^n\rtimes \mathfrak{S}_n $. 

   \begin{Prop}[Geck-J.]\label{formula}
Let $\ulambda,\umu\in \Pi^2_n$ and assume that  $N\in \mathbb{N}$ 
 is admissible for both $\la$ and $\mu$. Then we have 
 $$\kappa_{(b,N)} (\ulambda)\unlhd\kappa_{(b,N)} (\umu)\Longrightarrow \ua(E^\ulambda)    \geq\ua(E^\umu).$$
\end{Prop}

It is then natural to ask if the pre-order defined  by the dominance order on the  
$\kappa_{(b,N)} (\ulambda)$  coincide with $\preceq_L$.  This is the main result of this paper.

   \begin{Th}\label{main}
 Let $\ulambda=(\lambda^1,\lambda^2)\in \Pi^2_n$ and assume that $N\in \mathbb{N}$ 
 is admissible. Then we have 
 $$\kappa_{(b,N)} (\ulambda)\unlhd\kappa_{(b,N)} (\umu)\iff \ulambda\preceq_L \umu.$$
   \end{Th}
   Note that the implication 
    $$\kappa_{(b,N)} (\ulambda)\unlhd\kappa_{(b,N)} (\umu)\Longleftarrow \ulambda\preceq_L \umu$$
has already been proved in \cite[Theorem 7.11]{GL}. So this article is devoted to the proof of the reverse implication. 
 Note that the result has also been established in the case where:
 \begin{itemize}
 \item $b=0$ and $b=1$ by Geck and Iancu \cite[Ex. 8.2]{GL} using results by Spaltenstein \cite{Sp},
 \item $b>n-1$ by Geck and Iancu  \cite[\S 6]{GL}.
 \end{itemize}


\section{On the adjacency of Lusztig symbols with respect to $\tle$}
As we have seen in the previous section, to any bipartition $\ula$ of $n$ we can associate the Lusztig's symbol $\fB_{(b,N)}(\ula)$ and a partition $\ka_{(b,N)}(\ula)$. Therefore, the usual dominance order $\tle$ on partition yields an order on Lusztig's symbols and bipartitions of $n$. We will still denote this order by $\tle$. In this section we study adjacent bipartitions for this order. Let us first clarify what we mean by adjacent: Let $\ula$ and $\umu$ be two bipartitions of $n$ and let $N\in\nN$ be admissible for both $\ula$ and $\umu$. We say that $\ula\lhd\umu$ are adjacent if  $\kappa_{(b,N)} (\ulambda) \lhd \kappa_{(b,N)}  (\umu)$ and there is no bipartition $\unu$ of $n$ such that 
$\kappa_{(b,N)}  (\ulambda) \lhd \kappa_{(b,N)}  (\unu)  \lhd \kappa_{(b,N)}  (\umu)$.  Equivalently, if a bipartition $\unu$ satisfies 
$$\kappa_{(b,N)}  (\ulambda) \tle \kappa_{(b,N)}  (\unu)  \tle \kappa_{(b,N)}  (\umu)$$
then we have either $\unu=\ula$ or $\unu=\umu$. \\

It is a well-known fact, see for example \cite[(1.16)]{Mac-Do}, 
that if two partitions are adjacent for the dominance order then one can be obtain from the other by moving a single box in their Young diagram. 
The aim of this section is to show that a similar result holds for the dominance order on bipartitions and symbols: if $\ula\tle\umu$ are adjacent bipartitions then the partition 
$\ka_{(b,N)} (\ula)$ can be obtained from $\ka_{(b,N)} (\umu)$ moving a single box in their Young diagram.\\


\subsection{Sympartitions and symbols}
 In this section, we study the following problem: given a partition $\la$, under which condition can we 
 find a bipartition $\ula\in \Pi^2_n$ such that $\la=\fB_{(b,N)}(\ula)$ for some $b,N\in\nN$ ? 
First, we need to introduce more definitions. 
  
  \begin{Def}
 A $\ell${\it -overlap} in 
  a partition $\lambda$  is a repetition of exactly $\ell$ elements in $\lambda$.   
  \end{Def}
For example, there is one $2$-overlap in the 
   partition $\lambda=(4,4,2,2,2,1)$ and one $3$-overlap. 

\begin{Def} 
Let $b,N,n\in\nN^\ast$. We say that $\lambda$ is a $(b,N,n)$-sympartition if and only if 
\begin{itemize}
\item $\la$ is a partition of $f(b,N,n)$,
\item there are no $3$-overlap in $\lambda$,
\item the number of $2$-overlaps is at most $N$.  
\end{itemize}
When the triplet $(b,N,n)$ is clear from the context we will write sympartition instead of $(b,N,n)$-sympartition.  
\end{Def}

\begin{Exa}
\label{ex1}
The partition $(7,4,4,3,2,1,1,0)$ is a $(2,3,9)$-sympartition and also a $(4,2,6)$-sympartition. However it can't be a sympartition associate to a triplet of the form $(b,1,n)$ as there are two $2$-overlaps.
\end{Exa}
Recall the definition of admissible in Section \ref{type-B}.
\begin{Prop}\label{symp} Let $\ulambda$ be a bipartition of $n$ and let $N\in\nN$ be admissible for $\ula$. Then the partition $\kappa_{(b,N)} (\ulambda)$ is a $(b,N,n)$-sympartition. 
Conversely, let $\lambda$ be a $(b,N,n)$-sympartition. Then there exists a bipartition $\ulambda$ of $n$ such that 
 $\kappa_{(b,N)} (\ulambda)=\lambda$.
\end{Prop}
\begin{proof}
Let $\ulambda$ be a bipartition of $n$. Then by definition of the $(b,N)$-symbol, the sequences $(\mathfrak{B}^{1}_{i})_{i=1,\ldots,N+b}$ and $(\mathfrak{B}^{2}_{i})_{i=1,\ldots,N}$  are strictly increasing. It follows easily that there can't be any $3$-overlaps in $\kappa_{(b,N)} (\ulambda)$ and that the number of $2$-overlaps is less  than or equal to $N$. We have seen in Section \ref{type-B} that $\ka_{(b,N)}(\ula)$ is a partition of $f(b,N,n)$. Hence $\ka_{(b,N)}(\ula)$ is a $(b,N,n)$-sympartition. \\

 Assume now that $\lambda$ is a $(b,N,n)$-sympartition.
 Then there exist  two sequences of strictly increasing integers $\mathfrak{B}^1$
 and $\mathfrak{B}^2$ such that 
  \begin{itemize}
\item there are   $N$ elements in $\mathfrak{B}^1$ and  $N+b$ elements in $\mathfrak{B}^2$,
\item  $\mathfrak{B}^1 \cap \mathfrak{B}^2$ contains all the $2$-overlap of $\lambda$, 
 \item  $\mathfrak{B}^1 \cup \mathfrak{B}^2=\lambda$ (as a multiset). 
\end{itemize}
Now there exist two partitions $\lambda^1$ and $\lambda^2$ such that 
 the increasing  sequences $(\lambda^1_i-i+N+b)_{i=1,\ldots,N+b}$ and 
$(\lambda^2_i-i+N)_{i=1,\ldots,N}$  are respectively equal to $\mathfrak{B}^1$ 
 and $\mathfrak{B}^2$.  If we set $\ula:=(\la^1,\la^2)$ then one can check that $\la$ is a bipartition of $n$ and that  
 $$\mathfrak{B}_{(b,N)} (\ulambda)=\left(\begin{array}
[c]{l}
\mathfrak{B}^{2} \\ 
\mathfrak{B}^{1}     \\  
\end{array}
\right) \quand \kappa_{(b,N)} (\ulambda)=\lambda.$$ 
\end{proof}


\subsection{Raising operators}
Following \cite{Mac-Do}, we introduce the raising operators. To this end, we will work on $M$-uplet of integers instead of partitions or sympartitions. 
\begin{Def}
Let $a=(a_1,\ldots,a_M)\in\nZ^M$.  For $1\leq k_1<k_2\leq M$, we set 
\begin{align*}
\Up_{k_1,k_2}(a)&=(a_1,\ldots,a_{k_1}+1,\ldots,a_{k_2}-1,\ldots,a_M),\\
\Down_{k_1,k_2}(a)&=(a_1,\ldots,a_{k_1}-1,\ldots,a_{k_2}+1,\ldots,a_M).    
\end{align*}
\end{Def}
Let $\la$ be a partition of $n$ and assume that $\Up_{k_1,k_2}(\la)$ is also a partition of $n$. Note that this would be the case whenever $\la_{k_2}>\la_{k_2+1}$ and $\la_{k_1-1}>\la_{k_1}$. Then, looking at the associated Young tableau,  $\Up_{k_1,k_2}(\la)$ is obtained from $\la$ simply by moving the $(k_2,\la_{k_2})$-box to the $k_1\up{th}$-part of $\la$. 

\begin{Lem}
\label{j-sit}
Let $\la\lhd\la'$ be two partitions of $n$ and let 
$$i=\min\{k\in\nN\mid \la_k<\la'_k\}\quand j:=\min\{m\in\nN\mid m>i, \sum_{k=1}^{m}\la_k=\sum_{k=1}^{m}\la'_k\}.$$
We have $\la_j>\la'_j\geq \la'_{j+1}\geq \la_{j+1}$. Further, for all $i\leq k_1< k_2\leq j$ such that $ \Up_{k_1,k_2}(\la)$ is a partition of $n$ we have 
$$\la\lhd \Up_{k_1,k_2}(\la)\unlhd \la'.$$
Similarly, for all $i\leq k_1< k_2\leq j$ such that $ \Down_{k_1,k_2}(\la')$ is a partition of $n$ we have 
$$\la\unlhd \Down_{k_1,k_2}(\la')\lhd \la'.$$

\end{Lem}
\begin{proof}
By minimality of $j$ we know that
$$\sum_{k=1}^{j-1}\la_k<\sum_{k=1}^{j-1}\la'_k$$
therefore to have $\sum_{k=1}^{j}\la_k=\sum_{k=1}^{j}\la'_k$ we must have $\la_j>\la'_j$. The fact that $\la'_j\geq \la'_{j+1}$ is clear since $\la'$ is a partition. Finally we need to have $\la_{j+1}\leq \la'_{j+1}$ and this follows from the fact that  $\la\lhd \la'$.\\

Let $k_1,k_2\in\nN$ be such that $i\leq k_1\leq k_2\leq j$ and assume that $\Up_{k_1,k_2}(\la)$ is a partition of $n$. 
It is clear that $\la\lhd \Up_{k_1,k_2}(\la)$ since we moved a box in the upward direction. Let us show that  $\la'':=\Up_{k_1,k_2}(\la)\unlhd \la'$. Let $i\leq m\leq j$. Note that $\la''_m=\la_m$ for all $m<k_1$ thus
$$\sum_{k=1}^{m}\la_k=\sum_{k=1}^{m}\la''_k\quad\text{for all $m<k_1$}.$$
Further, for all $m\geq k_2$ we have
$$\sum_{k=1}^{m}\la_k=\sum_{k=1}^{m}\la''_k.$$
From those two inequalities and since $\la\unlhd \la'$, we get 
$$\sum_{k=1}^{m}\la''_k\leq \sum_{k=1}^{m}\la'_k\quad\text{for all $m<k_1$ and all $m\geq k_2$}.$$
Let $k_1\leq m< k_2$. By minimality of $j$ and since $m<j$, we have
$$\sum_{k=1}^{m}\la_k<\sum_{k=1}^{m}\la'_k.$$
By definition of $\la''$, we know that
$$\sum_{k=1}^{m}\la''_k=\sum_{k=1}^{m}\la_k+1$$
and therefore we get
$$\sum_{k=1}^{m}\la''_k\leq\sum_{k=1}^{m}\la'_k.$$
The result follows. The proof of the second part of the lemma is similar. 
\end{proof}


\subsection{A first simplification} 
Let $\la$ be a partition of length $r$. By convention, if $i>r$ we set $\la_i=0$ and if $i\leq 0$ we set $\la_i=+\infty$. For $i<j\in\nN$ we set 
$$\la_{[i,j]}=(\la_i,\ldots,\la_j), \la_{\leq i}=(\la_1,\ldots,\la_i)\quand \la_{\geq i}=(\la_i,\ldots,\la_r).$$
We define $\la_{<i},\la_{>i}$ in a similar and obvious fashion. 
\begin{Prop}  
\label{first-simple}
Let $\ula$ and $\umu$ be two bipartitions of $n$ and let $N\in\nN$ be admissible for $\ula$ and $\umu$. Assume that  $\ula\lhd\umu$ are adjacent. 
Let $\kappa:=\kappa_{(b,N)}  (\ulambda)=(\kappa_1,\ldots,\kappa_r)$ and  $\kappa':=\kappa_{(b,N)}  (\umu)=(\kappa_1 ',\ldots,\kappa_r')$ and set 
$$i=\min\{k\in\nN\mid \ka_k<\ka'_k\}\quand j:=\min\{m\in\nN\mid m>i, \sum_{k=1}^{m}\ka_k=\sum_{k=1}^{m}\ka'_k\}.$$
Then we have $\kappa_{<i} = {\kappa'}_{< i} $
   and $\kappa_{> j} = {\kappa'}_{> j}$.
\end{Prop}
\begin{proof}
It is clear that $\kappa_{< i} = {\kappa'}_{<i} $ by definition of $i$. Let us show that  
$\kappa_{> j} = {\kappa'}_{> j} $. 
 First, by minimality of $j$ we have $\kappa_{j}>\kappa_{j}'$ and since $\ka\lhd \ka'$ we also have $ \kappa_{j+1} ' \geq \kappa_{j+1}$. (In particular, we have $\ka_j>\ka_{j+1}$).
This implies that both  $\alpha:=(\kappa_{\leq j},\kappa_{\geq j+1} ')$ and  $\beta:=(\kappa_{\leq{j}} ',\kappa_{\geq j+1} )$ are partitions. In addition 
 we have 
 $$\begin{array}{rcl}
 |\alpha| & = & \sum_{1\leq k \leq j} \kappa_{k} +  \sum_{j+1\leq k \leq r} \kappa_{k} ' \\
 &=&  \sum_{1\leq k \leq j} \kappa_{k}'  +  \sum_{j+1\leq k \leq r} \kappa_{k} ' \\
&=& |\kappa'|
\end{array}$$
and 
 $$\begin{array}{rcl}
 |\beta| & = & \sum_{1\leq k \leq j} \kappa'_{k} +  \sum_{j+1\leq k \leq r} \kappa_{k}  \\
 &=&  \sum_{1\leq k \leq j} \kappa_{k} +  \sum_{j+1\leq k \leq r} \kappa_{k}  \\
&=& |\kappa|
\end{array}$$
So $\alpha$ and $\beta$ are both partitions of $f(b,N,n)$.
   We claim that either $\alpha$ or $\beta$ is a  $(b,N,n)$-sympartition. First let us show that 
 there are no $3$-overlaps in these two partitions: 
\begin{itemize}
\item Since $\kappa_{j}> \kappa_{j}'\geq \ka'_{j+1}$, it is easy to see that $\al$ cannot contain any $3$-overlap as $\ka$ and $\ka'$ doesn't by Proposition \ref{symp}.
\item Assume that there is a $3$-overlap in $\beta$. Then we must have $\kappa'_{j}=\kappa_{j+1}$ and either $\ka'_{j-1}=\ka'_j$ or $\ka_{j+2}=\ka_{j+1}$. First 
the fact that $\kappa'_{j}=\kappa_{j+1}$ implies that $\kappa'_{j}=\kappa_{j+1}=\ka'_{j+1}$ since $\ka'_j\geq \kappa_{j+1}'\geq \ka_{j+1}$. Second, since  $\kappa'_{j+1}=\kappa_{j+1}$ we must have $\ka'_{j+2}\geq \ka_{j+2}$ as  $\kappa \lhd \kappa'$. Finally, we get $\ka_{j+1}=\ka'_{j+1}\geq \kappa_{j+2} ' \geq \kappa_{j+2}$.

 \begin{itemize}
 \item Assume that $\kappa_{j-1} '=\kappa_{j}'$. Then 
 $\kappa_{j-1} '=\kappa_{j} '=\kappa_{j+1}'$
   and we have a $3$-overlap in $\kappa'$ which is a contradiction.
 \item Assume that  $\kappa_{j+1} =\kappa_{j+2}$.  Then the inequality $\ka_{j+1}=\ka'_{j+1}\geq \kappa_{j+2} ' \geq \kappa_{j+2}$
implies that $\ka'_{j+1}=\ka'_{j+2}$ and since we have seen that $\ka'_{j}=\ka'_{j+1}$ we have a $3$-overlap in $\kappa'$ which is a contradiction.
 \end{itemize}
\end{itemize}
Thus, as claimed, we conclude that we there is no $3$-overlap in both $\alpha$ and 
$\beta$. 

For $\unu$ a partition, we denote by $\fO(\unu)$ the number of $2$-overlaps in $\unu$. Then we have 
$$\begin{array}{rcl}
\fO(\kappa)&=&\fO (\kappa_{<j+1})+\fO( \kappa_{>j})\quad (\text{ we have seen that $\ka_j>\ka_{j+1}$})\\
\fO(\kappa')&=&\fO (\kappa_{<j+1} ')+\fO( \kappa_{>j}') +\delta_{\kappa_{j} ', \kappa_{j+1} ' }\\
\end{array}$$
where $\delta$ stands for  Kronecker delta. We also have 
$$\begin{array}{rcl}
\fO(\alpha)&=&\fO(\kappa_{<j+1})+\fO( \kappa_{>j} ')\\
\fO(\beta)&=&\fO(\kappa_{<j+1} ')+\fO( \kappa_{>j}) +\delta_{\kappa_{j} ', \kappa_{j+1}  }.\\
\end{array}$$
Note that if $\delta_{\kappa_{j} ', \kappa_{j+1} }=1$ then we also have 
$\delta_{\kappa'_{j} , \kappa'_{j+1}}=1$ thus in any case 
$\delta_{\kappa'_{j}, \kappa_{j+1} }\leq \delta_{\kappa'_{j}, \kappa'_{j+1}  }$. We want to show that 
 either $\fO(\alpha)\leq N$ or  $\fO(\beta) \leq N$.  We argue by contradiction assuming that 
 $\fO(\alpha)>N$ and $\fO(\beta) >N$.  Then 
$$\fO(\kappa)+\fO(\kappa')\geq \fO(\alpha)+\fO(\beta)\geq N+2$$
so 
$$\fO(\kappa)+\fO(\kappa')\geq 2N+2$$
 which is a contradiction because both $\kappa$ and $\kappa'$ are $(b,N,n)$-sympartitions and so we have 
 $\fO(\kappa)\leq N$ and  $\fO(\kappa ') \leq N$. Since neither $\al$ nor $\beta$ contain a $3$-overlap, we see that one of these partitions is a $(b,N,n)$-sympartition. \\

Finally, the hypotheses imply that $\kappa_{\leq j} \lhd {\kappa'}_{\leq j}$ (they are both partitions of the same rank) 
 and $\kappa_{> j} \unlhd {\kappa'}_{> j}$ (they are both partitions of the same rank).  
Thus we are in one of the following configuration:
\begin{itemize}
\item either $\alpha$ is a $(b,N,n)$-sympartition and we have 
 $$\kappa  \unlhd \alpha  \unlhd \kappa'  . $$
Then either $\kappa = \alpha$ or $\kappa'=\al$ and thus $\kappa_{>j}=\kappa_{>j} '$ since we have  $\kappa_{\leq j}\neq \kappa_{\leq j} '$.
\item  or $\beta$ is a $(b,N)$-sympartition and we have 
 $$\kappa  \unlhd \beta  \unlhd \kappa' .$$
Then either  $\kappa = \beta$ or $\kappa'=\beta$ and thus $\kappa_{>j}=\kappa_{>j} '$ since we have    $\kappa_{\leq j}\neq \kappa_{\leq j} '$.
 
\end{itemize}
This concludes the proof.

\end{proof}


\subsection{The double break Lemma }
From now on and until the end of this section, we will assume that  $\ula\lhd\umu$ are two adjacent bipartitions of $n$ and  that $N\in\nN$ is admissible for $\ula$ and $\umu$. We set  
$$\kappa:=\kappa_{(b,N)}  (\ulambda)=(\kappa_1,\ldots,\kappa_r), \kappa':=\kappa_{(b,N)}  (\umu)=(\kappa_1 ',\ldots,\kappa_r')$$
$$i=\min\{k\in\nN\mid \ka_k<\ka'_k\}\quand j:=\min\{m\in\nN\mid m>i, \sum_{k=1}^{m}\ka_k=\sum_{k=1}^{m}\ka'_k\}.$$
For a partition $\nu$ we set $\fJ^{\nu}_k:=\nu_k-\nu_{k+1}\geq 0$. We then say that $k$ is a break point of $\nu$ if and only if we have 
$\fJ^{\nu}_{k-1}\geq 1$ and $\fJ^{\nu}_k\geq 1$.

\Yboxdim10pt

\begin{Lem} Assume that for all $i\leq m\leq j-1$ we have $\fJ^{\ka'}_m=0$ or $1$. Then $\ka'$ has at least two break points $k_1$ and $k_2$ in between $i$ and $j-1$.
\end{Lem}
\begin{proof}
Assume that $\ka'$ has no break point. There are 4 cases to consider, whether $\fJ^{\ka'}_i$ and $\fJ^{\ka'}_{j-1}$ are equal to 0 or 1. Those four cases correspond to the following shapes for $\ka'$
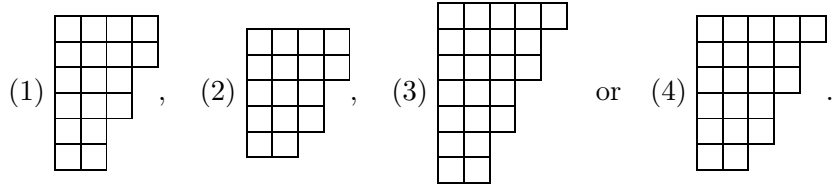
$$(1)\ \yng(4,4,3,3,2,2),\quad (2)\ \yng(4,4,3,3,2),\quad (3)\ \yng(5,4,4,3,3,2,2)\quor (4)\ \yng(5,4,4,3,3,2).$$
Let $\ell\in\nN$ be such that $i-j+1=2\ell+\eps$ where $\eps=0$ or $1$. 
A straightforward computation in each case yields 
\begin{align*}
(1)\quad& \sum_{k=i}^{j}\ka'_k=\sum^{\ell}_{k=1} 2(\ka'_i-k+1)=2\ell\ka_i'-\ell(\ell-1),\\
(2)\quad& \sum_{k=i}^{j}\ka'_k=\sum^{\ell}_{k=1} 2(\ka'_i-k+1)+(\ka'_i-\ell)=(2\ell+1)\ka_i'-\ell^2,\\
(3)\quad& \sum_{k=i}^{j}\ka'_k=\ka'_i+\sum^{\ell}_{k=1} 2(\ka'_i-k)=(2\ell+1)\ka_i'-\ell(\ell+1),\\
(4)\quad& \sum_{k=i}^{j}\ka'_k=\ka'_i+\sum^{\ell-1}_{k=1} 2(\ka'_i-k)+(\ka'_i-\ell)=2\ell\ka_i'-\ell^2.\\
\end{align*}
Form there, we see that the smallest value of $\sum_{k=i}^{j}\ka'_i$ that can be achieved when there are no break-point is  
$$\begin{array}{ccccccc}
2\ell\ka_i'-\ell^2& \text{ if $i-j+1$ is even,}\\
&\\
(2\ell+1)\ka_i'-\ell(\ell+1)& \text{ if $i-j+1$ is odd.}
\end{array}$$
Now the largest sympartition that one can construct starting with a $i^{\text{th}}$-part equal to $\ka_i<\ka'_i$ will satisfy
\begin{align*}
\sum_{k=i}^{j}\ka_k=
\begin{cases}
2\ell\ka_i-\ell(\ell-1)&\mbox{ if $i-j+1$ is even,}\\
(2\ell+1)\ka_i-\ell^2&\mbox{ if $i-j+1$ is odd.}
\end{cases}
\end{align*}
If $i-j+1$ is even then we would have
\begin{align*}
\sum_{k=i}^{j}\ka_k&\leq 2\ell\ka_i-\ell(\ell-1)\\
&\leq 2\ell(\ka'_i-1)-\ell(\ell-1)\\
&=2\ell\ka'_i-\ell(\ell+1)\\
&<\sum_{k=i}^{j}\ka'_i
\end{align*}
contradicting the fact that $\sum_{k=i}^{j}\ka'_i=\sum_{k=i}^{j}\ka_i$. \\
If $i-j+1$ is odd then we would have
\begin{align*}
\sum_{k=i}^{j}\ka_i&\leq(2\ell+1)\ka_i-\ell^2\\
&\leq (2\ell+1)(\ka'_i-1)-\ell^2\\
&= (2\ell+1)\ka'_i-(\ell+1)^2\\
&<\sum_{k=i}^{j}\ka'_i
\end{align*}
once again contradicting the fact that $\sum_{k=i}^{j}\ka'_i=\sum_{k=i}^{j}\ka_i$. So we see that $\kappa'$ has at least one break point. \\

Assume that $\ka'$ has a unique break point. Let $N$ be the length of the part corresponding to the breakpoint. Note that we have $\ka'_j<N<\ka'_i$. There are 4 cases to consider, whether $\fJ^{\ka'}_i$ and $\fJ^{\ka'}_{j-1}$ are equal to 0 or 1. Those four cases correspond to the following shape for $\ka'$
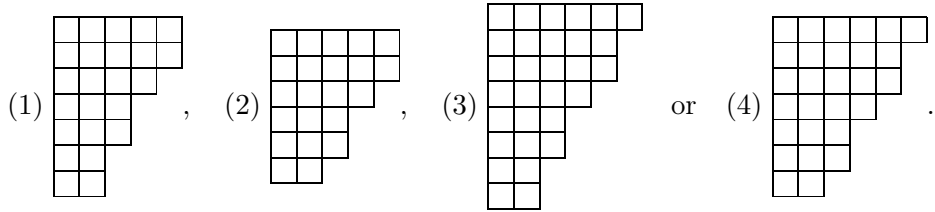
$$(1)\ \yng(5,5,4,3,3,2,2),\quad(2)\  \yng(5,5,4,3,3,2),\quad(3)\  \yng(6,5,5,4,3,3,2,2)\quor(4)\  \yng(6,5,5,4,3,3,2).$$
To use the computation done before, we will set $\ell$ to be such that $i-j+2=2\ell+\eps$ where $\eps=0$ or $1$. A straightforward computation in each case yields 
\begin{align*}
(1)\quad& \sum_{k=i}^{j}\ka'_k=2\ell\ka_i'-\ell(\ell-1)-N,\\
(2)\quad& \sum_{k=i}^{j}\ka'_k=(2\ell+1)\ka_i'-\ell^2-N,\\
(3)\quad& \sum_{k=i}^{j}\ka'_k=(2\ell+1)\ka_i'-\ell(\ell+1)-N,\\
(4)\quad& \sum_{k=i}^{j}\ka'_k=2\ell\ka_i'-\ell^2-N.\\
\end{align*}
From there, using the fact that $\ka'_j<N<\ka'_i$ we see that the smallest value of $\sum_{k=i}^{j}\ka'_k$ that can be achieved when there is exactly one break-point is  striclty less than
$$\begin{array}{ccccccc}
2\ell\ka_i'-\ell(\ell+1)& \text{ if $i-j+1$ is even,}\\
&\\
(2\ell-1)\ka_i'-\ell^2& \text{ if $i-j+1$ is odd.}
\end{array}$$
Now the largest sympartition that we can construct starting with a $i^{\text{th}}$-part equal to $\ka_i<\ka'_i$ satisfies
\begin{align*}
\sum_{k=i}^{j}\ka_k\leq
\begin{cases}
2\ell\ka_i-\ell(\ell-1)&\mbox{ if $i-j+1$ is even,}\\
(2\ell-1)\ka_i-(\ell-1)^2&\mbox{ if $i-j+1$ is odd.}
\end{cases}
\end{align*}
If  $i-j+1$ is even then we would have
\begin{align*}
\sum_{k=i}^{j}\ka_k&\leq 2\ell\ka_i-\ell(\ell-1)\\
&\leq 2\ell(\ka'_i-1)-\ell(\ell-1)\\
&=2\ell\ka'_i-\ell(\ell-1)-2\ell\\
&=2\ell\ka'_i-\ell(\ell+1)\\
&<\sum_{k=i}^{j}\ka'_k
\end{align*}
contradicting the fact that $\sum_{k=i}^{j}\ka'_k=\sum_{k=i}^{j}\ka_k$. \\
If $i-j+1$ is odd then we would have
\begin{align*}
\sum_{k=i}^{j}\ka_k&\leq(2\ell-1)\ka_i-(\ell-1)^2\\
&\leq (2\ell-1)(\ka'_i-1)-(\ell-1)^2\\
&= (2\ell-1)\ka'_i-(\ell-1)^2-(2\ell-1)\\
&=(2\ell-1)\ka'_i-\ell^2\\
&<\sum_{k=i}^{j}\ka'_k
\end{align*}
again contradicting the fact that $\sum_{k=i}^{j}\ka'_k=\sum_{k=i}^{j}\ka_k$. This conclude the proof. 
\end{proof}


\subsection{Adjacency of two sympartitions}

\begin{Th}
Let $b,N,n\in\nN$ and let $\ka,\ka'$ be two adjacent $(b,N,n)$-sympartitions with respect to $\lhd$. Then there exists $k_1,k_2\in\nN$ such that $\ka'=\Up_{k_1,k_2}(\ka)$. 
\end{Th}
\begin{proof}
The proof of this theorem is rather long and tedious and require a case by case analysis. First we set 
$$i:=\min\{k\in\nN\mid \ka_k<\ka'_k\}\quand j:=\min\{m\in\nN\mid m>i, \sum_{k=1}^{m}\ka_k=\sum_{k=1}^{m}\ka'_k\}.$$
Then by Proposition \ref{first-simple} we know that $\ka_{<i}=\ka_{<i}$ and $\ka_{>j}=\ka'_{>j}$. 
The idea of the proof is to construct a sympartition $\ka''$ either of the form $\Up_{k_1,k_2}(\ka)$ or $\Down_{k_1,k_2}(\ka')$ for some $i\leq k_1\leq k_2\leq j$. Assume that we have constructed such a sympartition $\ka''$. Then by Lemma \ref{j-sit}, we know that 
$$\ka\lhd \ka''\unlhd \ka' \quor \ka\unlhd \ka''\lhd \ka'.$$
Since $\ka,\ka'$ are adjacent this implies that $\ka''=\ka'$ or $\ka''=\ka$ as required. \\

{\bf Case 1.}  $\fJ^\ka_{i-1}\geq 2$ and $\fJ^{\ka}_j\geq 2$. In that case, we do not create any overlap by removing the $(\ka_j,j)$-box in $\ka$ nor by adding a box to the $i\up{th}$ part of $\ka$. Then if we set $\ka''= \Up_{i,j}(\ka)$ we have $\fO(\ka'')\leq \fO(\ka)$ and the result follows.  (Recall that $\fO(\nu)$ denotes the number of 2-overlaps in $\nu$.)\\

{\bf Case 2.} Assume that  $\fJ^\ka_{i-1}\geq 2$ and $\fJ^{\ka}_j=1$.  We have $\ka_j=\ka_{j+1}+1$ and by Lemma \ref{j-sit} we get  $\ka'_j=\ka'_{j+1}= \ka_{j+1}$. Then the shapes of $\ka$ and $\ka'$ are described in Figure \ref{ka1}.  
\begin{figure}[h]
        \centering
        \begin{subfigure}[b]{0.3\textwidth}
                \centering
$\yng(9,7,0,0,0,4,3)$
                \caption{Sympartition $\ka[i-1,j+1]$}
        \end{subfigure}
        \qquad\qquad
                        \begin{subfigure}[b]{0.3\textwidth}
                \centering
              $\yng(9,0,0,0,0,3,3)$                
                \caption{Sympartition $\ka'[i-1,j+1]$}
        \end{subfigure}
              \caption{Shape of $\ka$ and $\ka'$}\label{A2}
                \label{ka1}
\end{figure}

Of course we may have $\ka_{i-1}-\ka_i>2$.  Since $\ka'$ is a sympartition, we have $\ka_{j+1}=\ka'_{j+1}>\ka'_{j+2}=\ka_{j+2}$.\\

{\it Subcase 1: There exists an overlap in $\ka[i,j]$.} We set
$$k:=\max\{m\in\nN\mid i\leq  m< j, \fJ^{\ka}_{m}=0\}.$$ 
Let us show that $\ka_{k+2}>\ka_{k+3}$. Since $\ka_{j+1}>\ka_{j+2}$ and  $\ka_{j}>\ka_{j+1}$, this is true if $k=j-1$ or $k=j-2$. 
If $k<j-2$, we cannot have $\ka_{k+2}=\ka_{k+3}$ since $k+2<j$ and this would contradict the maximality of $k$.

Consider the partition $\ka''= \Up_{k+1,i}(\ka)$. We claim that $\ka''$ is a sympartition. To see this, note that since $\ka_{k+2}>\ka_{k+3}$ and $\ka_{i-1}-\ka_i\geq 2$ there is no $3$-overlap in $\ka''$. Then by removing the $(k+1,\ka_{k+1})$-box, we eliminate the $2$-overlap at place $(k,k+1)$ while possibly creating one at  place $(k+1,k+2)$. Also, by adding a box at the $i\up{th}$-part of $\ka$, we cannot have created a $2$-overlap since $\ka_{i-1}-2\geq \ka_{i}$. So the number of $2$-overlaps in $\ka''$ is less than or equal to the number of $2$-overlaps in $\ka$. This shows that $\ka''$ is a sympartition and we get the result in this case. \\

{\it Subcase 2: There is no overlap between $i$ and $j$.} 
 Set $\ka''= \Up_{i,j}(\ka)$. 
First,  since $\ka_{j+1}>\ka_{j+2}$, we do not create a $3$-overlap removing the $(j,\ka_j)$-box in $\ka$.
Then we have  
\begin{align*}
\fO(\ka'')&=\fO(\ka[1,i-1])+1+\fO(\ka[j+2,\ell(\ka)])\\
&=\fO(\ka'[1,i-1])+1+\fO(\ka'[j+2,\ell(\ka')])\\
&\leq \fO(\ka')
\end{align*}
where $\ell(\ka)$ denotes the length of $\ka$. The last inequality holds because $\ka'$ has a $2$-overlap at place $(j,j+1)$ and may have more between $i$ and $j$. Since $\ka'$ is a sympartition, so is $\ka''$. \\

{\it From now on and until the end of the proof we will assume that $\fJ^\ka_{i-1}=1$. 
 In that case, since $\ka'_{i-1}=\ka_{i-1}$ and  $\ka'_{i}>\ka_{i}$, we have $\ka'_i=\ka'_{i-1}=\ka_i+1$. That is, the sympartitions $\ka$ and $\ka'$ have shapes as shown in Figure \ref{ka2}.
\begin{figure}[h]
        \centering
        \begin{subfigure}[b]{0.3\textwidth}
                \centering
$\yng(9,8)$
                \caption{Sympartition $\ka[i-1,i]$}
        \end{subfigure}
        \qquad\qquad
                        \begin{subfigure}[b]{0.3\textwidth}
                \centering
                $\yng(9,9)$                
                \caption{Sympartition $\ka'[i-1,i]$}
        \end{subfigure}
              \caption{Shape of $\ka$ and $\ka'$}
              \label{ka2}
\end{figure}}

{\bf Case 3.} $\fJ^\ka_{i}=0$ and $\fJ^\ka_{j}=1$. Note that by Lemma \ref{j-sit}, we have $\ka'_j=\ka'_{j+1}=\ka_{j+1}$. The shape of $\ka$ and $\ka'$ are described in Figure \ref{ka3}. 
\begin{figure}[h]
\centering
\begin{subfigure}[b]{0.3\textwidth}
\centering
\yng(9,8,8,0,0,0,5,4)
\caption{Sympartition $\ka[i-1,j+1]$}
\end{subfigure}
\qquad\qquad
\begin{subfigure}[b]{0.3\textwidth}
\centering
\yng(9,9,0,0,0,0,4,4)
\caption{Sympartition $\ka'[i-1,j+1]$}
\end{subfigure}
\caption{Shape of $\ka$ and $\ka'$}
\label{ka3}
\end{figure}

Let  $k:=\max\{m\in\nN\mid i\leq  m< j, \fJ^{\ka}_{m}=0\}$. Note that $k$ is well defined since there is a $2$-overlap at place $(i,i+1)$. Set $\ka''=\Up_{i,k+1}(\ka)$. 
Arguing as in Case~2.1, using the maximality of $k$,  
one can show that we do not create a $3$-overlap by removing the $(k+1,\ka_{k+1})$-box and that the number of $2$-overlaps remains constant.
Also, we have $\ka_{i-2}=\ka'_{i-2}>\ka'_{i-1}=\ka_{i-1}$, so that we can add a box at the $i\up{th}$-part of $\ka$ without creating a $3$-overlap while keeping the number of $2$-overlaps constant.  It follows that $\ka''$ is a sympartition as required. \\

{\bf Case 4: $\fJ^\ka_{i}=0$  and $\fJ^\ka_{j}>1$}. The shape of $\ka$ and $\ka'$ are described in Figure \ref{ka4}.
Of course, we may have $\ka_j-\ka_{j-1}> 2$. As in the previous case we have  $\ka_{i-2}>\ka_{i-1}$, so that we can add a box at the $i\up{th}$-part of $\ka$ without creating a $3$-overlap while keeping the number of $2$-overlaps constant.  Then setting $\ka''=\Up_{i,j}(\ka)$ easily yields the result. 
\begin{figure}[h]
\centering
\begin{subfigure}[b]{0.3\textwidth}
\centering
\yng(9,8,8,0,0,0,5,3)
\caption{Sympartition $\ka[i-1,j+1]$}
\end{subfigure}
\qquad\qquad
\begin{subfigure}[b]{0.3\textwidth}
\centering
\yng(9,9,0,0,0,0,0,0)\caption{Sympartition $\ka'[i-1,i]$}
\end{subfigure}
\caption{Shape of $\ka$ and $\ka'$}
\label{ka4}
\end{figure}

{\bf Case 5: $\fJ^\ka_{i}\geq 1$ and $\fJ^\ka_{j}=1$}. The shape of $\ka$ and $\ka'$ are described in Figure \ref{ka5}.
\begin{figure}[h]
\centering
\begin{subfigure}[b]{0.3\textwidth}
\centering
\yng(9,8,7,0,0,0,5,4)
\caption{Sympartition $\ka[i-1,j+1]$}
\end{subfigure}
\qquad\qquad
\begin{subfigure}[b]{0.3\textwidth}
\centering
\yng(9,9,0,0,0,0,4,4)
\caption{Sympartition $\ka'[i-1,i]$}
\end{subfigure}
\caption{Shape of $\ka$ and $\ka'$}
\label{ka5}
\end{figure}
We may have $\ka_i-\ka_{i-1}\geq 1$. Arguing as in Case 3, we can show that $\ka_{j+1}>\ka_{j+2}$.\\
First assume that there exists $i\leq m< j$ such that $\ka'_{m}-\ka'_{m+1}\geq 2$. Then set  
$$k:=\max\{\ell\in\nN\mid   \ell< m, \fJ^{\ka'}_{\ell}=0\}.$$ 
Note that $k$ is well-defined as there is an overlap in $\ka'$ at places $(i-1,i)$ and that we have $k\geq i-1$. Consider $\ka''=\Down_{k+1,m+1}(\ka'')$. We show that there is no $3$-overlap in $\ka''$. If $k<m-1$ then we cannot have $\ka'_{k+1}=\ka'_{k+2}$ because this would contradicts the maximality of $k$.  If $k=m-1$ then we do not create an $3$-overlap since $\ka'_{m}-\ka'_{m+1}\geq 2$.  We show that $\fO(\ka'')\leq \fO(\ka')$. If $k<m-1$, then by removing the $(k+1,\ka'_{k+1})$-box from $\ka'$, we remove an overlap at place $(k,k+1)$ while possibly creating one at place $(k+1,k+2)$. By adding a box at the $m+1\up{th}$-part of $\ka''$ we do not create a $2$-overlap at place $(m,m+1)$ since $\ka'_{m}-\ka'_{m+1}\geq 2$, hence the result in this case.  If $k=m-1$ then we are in the situation described in Figure \ref{km} and the result follows easily.
\begin{figure}[h]
\centering
$\yng(6,6,4)\quad\leadsto\quad\yng(6,5,5)$
\caption{From $\ka'[k,k+2]$ to $\ka''[k,k+2]$}
\label{km}
\end{figure}

Second assume that for all $i\leq m<  j$ we have $\fJ_{m}^{\ka'}=0$ or $1$. 
Then by the double break Lemma, there exist at least two break points in $\ka'$. Let $k_1$ (respectively $k_2$) be the highmost (respectively lowest) one. Note that it is necessarily preceeded (respectively followed) by an overlap. 
Then set $\ka''=\Down_{k_1-1,k_2+1}(\ka')$. It can be seen in Figure \ref{ka8} that $\fO(\ka'')=\fO(\ka')$, so that $\ka''$ is a sympartition and the result follows. \\

{\bf Case 6: $\fJ^\ka_{i}\geq 1$ and $\fJ^\ka_{j}> 1$}. In that case, we can remove the $(j,\ka_j)$-box of $\ka$ without creating a $2$-overlap nor a $3$-overlap. If there exists $i\leq k<j-1$ such that $\fJ^\ka_{k-1}\geq 2$ then we set  $\ka''=\Up_{k,j}(\ka)$ and $\ka''$ is a sympartition. Otherwise, there are no such $k$ and $\fJ^\ka_{k-1}=0$ or $1$ for all $i\leq k<j-1$. If between $i$ and $j$ there is a sequence of shape 
$$(\ast)\quad \yng(8,7,6,6)$$
then we can put the $(j,\ka_j)$-box at the $3\up{rd}$-line starting from the top to obtain the desired sympartition $\ka''$. Since, the top of $\ka$ has the following shape 
$$\yng(9,8,7)$$
we see that either $\ka[i,j-1]$ is a "staircase" or one can find a shape as in $(\ast)$. If it is a staircase then we set $\ka''=\Up_{i,j}(\ka)$. This situation is described in Figure \ref{ka9} in the case where $\fJ^\ka_{j-1}=1$ though it is possible to have $\fJ^\ka_{j-1}\geq 1$. The number of overlap in $\ka''$ is then
\begin{align*}
\fO(\ka'')&=\fO(\ka[1,i-1])+1+\fO(\ka[j+1,\ell(\ka)])\\
&= \fO(\ka'[1,i-1])+1+\fO(\ka'[j+1,\ell(\ka')])\\
&\leq \fO(\ka')
\end{align*}
and thus $\ka''$ is a sympartition. This concludes the proof of the theorem. \\

\begin{figure}[h]
\centering
\begin{subfigure}[b]{1\textwidth}
\centering
$\yng(12,12,11,11,10,9,0,0,7,6,5,5,4,4)\leadsto\quad \yng(12,12,11,10,10,9,0,0,7,6,6,5,5,4)$
\caption{Case $k_1<k_2-1$}
\end{subfigure}
$\ $\\

\begin{subfigure}[b]{.9\textwidth}
\centering
$\yng(10,9,9,8,7,6,6,0,0)\leadsto\quad\yng(10,9,8,8,7,7,6,0,0)$
\caption{Case $k_1=k_2-1$}
\end{subfigure}
\caption{From $\ka'$ to $\ka''=\Down_{k_1-1,k_2+1}(\ka')$}
\label{ka8}
\end{figure}

\begin{figure}[h]
\centering
\centering
 $\yng(9,8,7,6,5,4,2)\leadsto \quad\yng(9,9,7,6,5,3,2)$
\caption{From $\ka$ to $\ka''$}
\label{ka9}
\end{figure}

\end{proof}


\section{Proof of the main result}
We start  by giving an explicit characterisation of the preorder $\preceq_L$ as in \cite{GL}, which is in spirit very similar to the one given in Example \ref{type-A} in type A. Then, using our characterisation of the adjacency of bipartitions we will prove the main result of this paper, that is, the order $\preceq_L$ on bipartition is the same as the order $\unlhd$ in the integer case.  

\subsection{On the pre-order $\preceq_L$ in type $B_n$}

We refer to  \cite[Lemma 7.5]{GL} for details in this section.  
Let $\ulambda$ and $\umu$ be two bipartitions of $n$. We have 
 $\ulambda \preceq_L \umu$ if and only if there exists a sequence 
 $$\ulambda:=\ulambda_0,\ulambda_1,\ldots, \ulambda_m:=\umu$$
 such that for each $i\in \{1,\ldots,m\}$, the following condition is satisfied:
 
 There exist a decomposition $n=k_i+l_i$ and  bipartitions $\unu_i$, $\unu_i '$ of $k_i$ such that  $E^{\unu_i} \preceq_L E^{\unu_i '}$   and
 \begin{itemize}
 \item either  $E^{\unu_i} \boxtimes \varepsilon_{l_i}\uparrow E^{ \ulambda_{i-1}}$ 
  and  $E^{\unu_i '} \boxtimes \varepsilon_{l_i} \rightsquigarrow_L E^{\ulambda_i}$ 
  \item or $E^{\unu_i} \boxtimes \varepsilon_{l_i}\uparrow E^{ \overline{\ulambda_{i}}}$ and $E^{\unu_i '} \boxtimes \varepsilon_{l_i}\rightsquigarrow E^{  \overline{\ulambda_{i-1}}}$ 
(where $\overline{\umu}$ means the transpose of $\umu$). 
\end{itemize}
\begin{Lem}\label{rec}
Let $n=k+l$ where $k\geq 0$ and $l\geq 1$. Let $\ulambda$ be a bipartition of $k$ and let $\umu$ be a bipartition of $n$. Then 
$E^{ \umu}$ is a constituent of  $\Ind(E^\ulambda \boxtimes \varepsilon_{l})$ 
 if and only if $\kappa_{(b,N)}  (\umu)$ can be obtained from $\kappa_{(b,N)}  (\ulambda)$ by increasing $l$ parts by $1$. In addition, we have 
 $\ba(E^\ulambda \boxtimes \varepsilon_{l})=\ba(E^\umu)$ if and only if these parts are the $l$ largest entries of  $\kappa_{(b,N)} (\ulambda)$.
\end{Lem}
\begin{proof}

The first assertion is proved in \cite[Lemma 7.6]{GL}. The ``if'' part of the second assertion is a result of Lusztig \cite[Lemma 7.10]{GL}.
 Let us show that if we increase the $l$ largest entries of  $\kappa_{(b,N)} (\ulambda)$ by 1 then  $\ba(E^\ulambda \boxtimes \varepsilon_{l})=\ba(E^\umu)$. 
  First, by \cite[Remark 2.8]{GL}, we have that  
  $$\ba(E^\ulambda \boxtimes \varepsilon_{l} )= \ba(E^\ulambda)+ \ba(\varepsilon_{l}).$$
  Next  by \cite[Ex. 1.3.8]{GJ}, we have that 
   $$ \ba_{\varepsilon_{l} } = \sum_{1\leq i\leq l} (i-1).$$
If $\kappa_{(b,N)}  (\umu)$ is  obtained from $\kappa_{(b,N)}  (\ulambda)$ by increasing the $l$ largest entries of  $\kappa_{(b,N)} (\ulambda)$ then, by Proposition \ref{formula}, we obtain
 $$\ba_{E^\umu}=\ba_{E^\ulambda}+  \sum_{1\leq i\leq l} (i-1)$$
as required. \end{proof}


\subsection{Proof of the main result}
We are now in position to prove Theorem \ref{main}.  
Assume that $\ulambda$ and $\umu$ are two bipartitions such that  $\kappa:=\kappa (\ulambda) \unlhd \kappa':=\kappa (\umu)$. 
Assume that $\kappa=\kappa'$ then we already know that $\ulambda$ and $\umu$ are in the same family 
 and thus that $\ulambda\preceq_L \umu$. So let us assume that $\kappa \lhd \kappa'$.

 We want to show that $\ulambda\preceq_L \umu$. For this, it is enough to consider the case where $\ka$ and $\ka'$ are adjacent.  In the previous section, we have seen that there exist $j_1$ and $i_1$ such that $\ka'=\Up_{i_1,j_1}(\ka)$. In other words
 $$\begin{array}{rcl}
 \kappa&=&(\ldots, \kappa_{i_1-1},\kappa_{i_1},\kappa_{i_1+1},\ldots,\kappa_{j_1-1},\kappa_{j_1},\kappa_{j_1+1},\ldots)\\
 \kappa'&=&(\ldots, \kappa_{i_1-1},\kappa_{i_1}+1,\kappa_{i_1+1},\ldots,\kappa_{j_1-1},\kappa_{j_1}-1,\kappa_{j_1+1},\ldots).
 \end{array}$$ 
Note that this implies that there is at most one element in $\kappa$ which is equal to $\kappa_{j_1}'=\kappa_{j_1}-1$. Indeed, since the $(j_1,\ka_{j_1})$-box can be remove in $\ka$, it implies that $\ka_{j_1+1}\leq \ka_{j_1}-1$. If the inequality is strict, then the result is obvious since no element of $\ka$ is equal to $\kappa_{j_1}-1$. Next if $\ka_{j_1+1}=\ka_{j_1}-1$ then we see that we must have $\ka_{j_1+2}<\ka_{j_1+1}$, otherwise we would create a $3$-overlap by removing the $(j_1,\ka_{j_1})$-box. \\


{\bf Case 1: $\ka_{j_1-1}\neq \ka_{j_1}$.}\\

We will construct a $(N,b,n-\ell)$-sympartition $\kappa''$ for some $\ell\in \mathbb{N}$  satisfying the two following properties:
 \begin{itemize}
 \item $\kappa$ is obtained from $\kappa''$ by increasing $l$ entries, 
  \item $\kappa'$ is obtained from $\kappa''$ by increasing the $l$ largest entries of $\kappa''$. 
 \end{itemize} 
 Consider the partition $\kappa''$ such that 
 $$\kappa_j ''=\left\{
 \begin{array}{ll}
 \kappa_j' -1  & \text{if $j< j_1$},\\
 \kappa_j' & \text{ if $j\geq j_1$}. \\
 \end{array}\right.$$
Note that this is indeed a partition since we have seen that there is at most one element in $\kappa$ which is equal to $\kappa_{j_1}'=\kappa_{j_1}-1$.
 In other words, we have
 $$\kappa''=(\ldots, \kappa_{i_1-1}'-1,\kappa_{i_1}'-1,\kappa_{i_1+1}'-1,\ldots,\kappa_{j_1-1}'-1,\kappa_{j_1} ',\kappa_{j_1+1}',\ldots)$$
and also
$$\kappa''=(\ldots, \kappa_{i_1-1}-1,\kappa_{i_1},\kappa_{i_1+1}-1,\ldots,\kappa_{j_1-1}-1,\kappa_{j_1}-1,\kappa_{j_1+1},\ldots).$$
We show that $\kappa''$ is a $(N,b,n-\ell)$-sympartition. It is clear that $\ka''$ is a partition of $f(N,b,n-\ell)$. So we only need to check that (1) there is no $3$-overlap and (2) the number of $2$-overlaps is less than or equal to $N$. 
 \begin{enumerate}
\item[(1)] There is no $3$-overlap in $\ka'$, thus the only possiblity to have created one in $\ka''$ is to have 
$\kappa_{j_1-1}'-1=\kappa_{j_1} '=\kappa_{j_1+1}'$ but this would imply that $\ka_{j_1-1}=\ka_{j_1}$, contradicting the hypothesis.

 \item[(2)]  The number of  $2$-overlaps in $\kappa''$ satisfies
  $$\begin{array}{rcl}
  \fO (\kappa'')&=&\fO({\kappa''}_{<j_1})+\delta_{{\kappa''}_{j_1-1},{\kappa''}_{j_1}}+\fO({\kappa''}_{\geq j_1})\\
  &=& \fO({\kappa'}_{<j_1})+\delta_{{\kappa'}_{j_1-1}-1,{\kappa'}_{j_1}}+\fO({\kappa'}_{>j_1}).
  \end{array}
  $$
  As we have  $\kappa_{j_1-1} '-1=\ka_{j_1-1}-1\neq \kappa_{j_1}-1=\kappa_{j_1} '$ by hypothesis we conclude that 
 $$\fO (\kappa'')\leq \fO (\kappa')\leq N.$$
 \end{enumerate}

 Now, the $j_1-1$ largest elements in $\kappa''$
  are the elements  $\kappa_i-1$  with  $i< j_1$. Thus $\kappa'$ can be obtained from  $\kappa''$ by adding $1$ to the $j_1-1$ largest entries
   where as $\kappa$ can be obtained from  $\kappa''$ by adding $1$ to $j_1-1$ elements. This implies that $\ulambda\preceq_L \umu$ as desired by Lemma \ref{rec}.\\

{\bf Case 2: $\ka_{j_1-1}= \ka_{j_1}$.}\\
 The partition $\ka$, $\ka'$ have the following shape around $j_1$.
 \begin{figure}[h]
\centering
\begin{subfigure}[b]{0.4\textwidth}
\centering
\yng(0,6,5,5,4,0)
\caption{Sympartition $\ka[j_1-2,j_1+1]$}
\end{subfigure}
\qquad\qquad
\begin{subfigure}[b]{0.4\textwidth}
\centering
\yng(0,6,5,4,4,0)
\caption{Sympartition $\ka'[j_1-2,j_1+1]$}
\end{subfigure}
\caption{Shape of $\ka$ and $\ka'$}
\label{ka5p}
\end{figure}

We consider the transposed bipartitions $\overline{\ulambda}$ and $\overline{\umu}$. 
We choose $t$ so that we have the following inclusion of multisets
$$(t-\kappa_1,\ldots, t-\kappa_r)\subseteq (0,1,\ldots,t,0,1,\ldots, t)$$
and
$$(t-\kappa'_1,\ldots, t-\kappa'_r)\subseteq (0,1,\ldots,t,0,1,\ldots, t).$$
Then the 
sympartitions $\ov{\ka}$ and $\ov{\ka}'$ associated respectively to $\ov\ula$ and $\ov\umu$ are just the partition obtained by reordering the complement of the first multisets into the seconds in the above inclusions; see \cite[\S 22]{Luequal}. There exist $i_2<j_2$ such that
$$\begin{array}{rcl}
   \overline{\kappa '}&=&(\overline{\kappa '}_1,\ldots,\overline{\kappa'}_{i_2-1},\overline{\kappa'}_{i_2},
   \overline{\kappa'}_{i_2+1},\ldots,\overline{\kappa'}_{j_2-1},\overline{\kappa'}_{j_2},
   \overline{\kappa'}_{j_2+1},\ldots)\\
   \overline{\kappa}&=&(\overline{\kappa'}_1,\ldots,\overline{\kappa'}_{i_2-1},\overline{\kappa'}_{i_2}+1,
   \overline{\kappa'}_{i_2+1},\ldots,\overline{\kappa'}_{j_2-1},\overline{\kappa'}_{j_2}-1,
   \overline{\kappa'}_{j_2+1},\ldots)\end{array}.$$
 Looking at the shape of $\ka$, $\ka'$ around $j_1$, we see that the shape of $\ov\ka$ and $\ov\ka'$ around $i_2$ is the following
  \begin{figure}[h]
\centering
\begin{subfigure}[b]{0.33\textwidth}
\centering
\yng(0,7,6,4,0)
\caption{Sympartition $\ka[i_2-1,i_2+1]$}
\end{subfigure}
\qquad\qquad
\begin{subfigure}[b]{0.33\textwidth}
\centering
\yng(0,7,5,4,0)
\caption{Sympartition $\ka'[i_2-1,i_2+1]$}
\end{subfigure}
\caption{Shape of $\ka$ and $\ka'$}
\label{ka5pp}
\end{figure}

Consider the partition $\ka''$ defined by
 $$\kappa_j ''=\left\{
 \begin{array}{ll}
 \kappa_j -1  & \text{if $j\leq i_2$},\\
 \kappa_j & \text{ if $j> i_2$.} \\
 \end{array}\right.$$
Since $\ka_{i_2}- \ka_{i_2+1}\geq 2$, we see that $\ka''$ is a sympartition. Then $\ka$ is obtained from $\ka''$ by adding one box to the $i_2$ largest parts and $\ka'$ is obtained by adding $i_2$ boxes. This concludes the proof of the Theorem.

\begin{Rem}
In the non-integer case, the definition of symbol is different (see \cite[Definition 3.1]{GL}) and it is not clear, at least to us, how to characterise adjacency of symbols in this case. 
\end{Rem} 


\noindent
{\small \textsc{ Jeremie Guilhot: Laboratoire de Math\'ematique et de Physique Th\'eorique, UMR 7350, 37000 Tours}\\
{\it Email address} \url{jeremie.guilhot@lmpt.univ-tours.fr}}

\medskip

\noindent
{\small \textsc{Nicolas Jacon : Laboratoire de Math\'ematique de Reims, 51687 REIMS cedex 2 }\\
{\it Email address} \url{nicolas.jacon@univ-fcomte.fr}}

\end{document}